\documentclass[12pt]{amsart}
\usepackage{dsfont}
\usepackage{enumitem}

\usepackage{amsmath}
\usepackage[foot]{amsaddr}
\usepackage{amssymb}
\usepackage{amsthm}
\usepackage{epsfig}
\usepackage{url}
\usepackage{array}
\usepackage{varwidth}
\usepackage{tikz}
\usetikzlibrary{patterns}
\usepackage{multirow}
\usepackage{textcomp}
\usepackage{bbm}
\usepackage{float}
\usepackage{color}

\newtheorem{theorem}{Theorem}

\newtheorem{corollary}[theorem]{Corollary}

\newcommand{\Fq}{\mathbb{F}_q}
\newcommand{\Fqn}{\mathbb{F}_q^n}

\title{Sets Avoiding Full-Rank Three-Point Patterns in $(\mathbb{F}_q^n)^k$ are Exponentially Small}
\author{Mohamed Omar}
\thanks{Harvey Mudd College, Department of Mathematics, Claremont, CA. Email: omar@g.hmc.edu. This work was supported by the AMS Claytor-Gilmer Fellowship and the Karen-EDGE Fellowship.}

\date{\today}
 
\begin{document}

\maketitle

\begin{abstract}
We prove that if a subset of  $(\mathbb{F}_q^n)^k$ (with $q$ an odd prime power) avoids a full-rank three-point pattern  $\vec{x},\vec{x}+M_1\vec{d},\vec{x}+M_2\vec{d}$ then it is exponentially small, having size at most $3 \cdot c_q^{nk}$ where $0.8414 q \leq c_q \leq 0.9184 q$. This generalizes a theorem of Kova\u{c} and complements results of Berger, Sah, Sawhney and Tidor. As a consequence, we prove that if $3$ is a square in $\Fq$ then subsets of $(\Fqn)^2$ avoiding equilateral triangles are exponentially small.
\end{abstract}


\section{Introduction}

Given $M_1,M_2,\ldots,M_{r-1} \in \Fq^{k \times k}$, the \emph{$r$-point pattern generated by $M_1,\ldots,M_{r-1}$} in $X=\Fqn$ is the set $S \subseteq X^k$  of all non-trivial $r$-tuples of the form $\vec{x}, \ \vec{x}+M_1 \vec{d}, \ \ldots, \ \vec{x}+M_{r-1} \vec{d}$ ranging over $\vec{x},\vec{d} \in X^k$, non-trivial here meaning the $r$ vectors are distinct. Understanding sizes of sets containing or avoiding specific $r$-point patterns is ubiquitous in the literature, many of the pertinent examples coming from the finite field model in arithmetic combinatorics. 

Indeed, a celebrated theorem Frustenberg and Katznelson \cite{furstenberg1979ergodic} which is a generalization of a classic theorem of Szemer\'{e}di states that if a set $A \subseteq \{1,2,\ldots,N\}^k$ avoids a nontrivial homothetic copy of a finite set $Y \subset \mathbb{Z}^k$, then $|A|=o(N^k)$. Effective bounds for specific sets $Y$ have served to be challenging, but a common avenue for insight comes from moving to a finite field model: asking if $A \subseteq X^k$ (here $X=\Fqn$) avoids a homothetic copy of a finite set $Y \subset X^k$, then what effective bounds are there on $|A|$ that can refine the generic $o(q^{nk})$? Such questions are relevant because homothetic copies of a given set can typically be written as $r$-point patterns. For instance, in a recent breakthrough of Peluse \cite{peluse2022subsets}, it was proven that if $A \subseteq (\Fqn)^2$ (with $q$ prime) avoids the four-point pattern generated by \[M_1=\begin{pmatrix} 0 & 0 \\ 0 & 1 \end{pmatrix}, \ M_2=\begin{pmatrix} 0 & 0 \\ 0 & 2 \end{pmatrix}, \ M_3=\begin{pmatrix} 0 & 1 \\ 0 & 0 \end{pmatrix},\] 
then $|A| \leq \frac{|X|^2}{\log_m |X|}$ for some large $m$. Viewing the four-point pattern as living in the $2$-dimensional set $X^2$, it is  the set of homothetic copies of 
\[\left \{ \begin{pmatrix} 0 \\ 0 \end{pmatrix}, \begin{pmatrix} 0 \\ 1 \end{pmatrix}, \begin{pmatrix} 0 \\ 2 \end{pmatrix}, \begin{pmatrix} 1 \\ 0 \end{pmatrix} \right \}\] 
in $(\Fqn)^2$.

Beyond arithmetic combinatorics and relating to a problem in Ergodic theory,  Berger, Sah, Sawhney and Tidor \cite{berger2022popular} studied \emph{full-rank} three-point patterns, that is three-point patterns $\vec{x},\vec{x}+M_1\vec{d},\vec{x}+M_2\vec{d}$ where $M_1,M_2 \in \Fq^{k \times k}$ and $M_1-M_2$ are all invertible. They proved that such patterns have a popular difference, meaning if $A \subseteq X^k$ with $|A| \geq \alpha \cdot |X|^k$ for $\alpha \in (0,1)$ then there is some $\vec{d} \neq \vec{0}$ so that $\left \{ \vec{x} \in X^k \ : \ \vec{x},\vec{x}+M_1\vec{d},\vec{x}+M_2\vec{d} \in A \right \}$ has size at least $(\alpha^3-o(1)) \cdot |X|^k$. Their result is in fact more general, with $X$ replaced by a general finite abelian group and $M_1,M_2,M_1-M_2$ being automorphisms of the group. This discovery generalized the work of Kova\u{c} \cite{kovac} who proved the existence of a popular difference for $X=\Fq^n$ with the specific three-point pattern generated by 
\begin{equation}\label{eqn:generators}
M_1=\begin{pmatrix} 1 & 0 \\ 0 & 1 \end{pmatrix}, \ M_2=\begin{pmatrix} 0 & -1 \\ 1 & 0 \end{pmatrix},
\end{equation}
which encompasses precisely all right-angled isosceles triangles in $(\mathbb{F}_q^n)^2$ (again viewed as $2$-dimensional here). 

In this same spirit, we generalize a complementary facet of Kova\u{c}'s work. In  \cite{kovac}, Kova\u{c} later proves that subsets of $X^k$ avoiding the three-point pattern generated by the matrices in (\ref{eqn:generators}) are exponentially small, with size at most $3 \cdot c_q^{2n}$ where $1 \leq c_q < q$. Our main theorem generalizes this result to arbitrary full-rank three-point patterns.

\begin{theorem}[Main Theorem]\label{thm:main}
If $A \subseteq (\Fqn)^k$ avoids a full-rank three-point pattern $\vec{x},\vec{x}+M_1\vec{d},\vec{x}+M_2\vec{d}$ then $|A| \leq 3 \cdot c_q^{kn}$ where $0.8414q \leq c_q \leq 0.9184q$.
\end{theorem}

A particularly interesting consequence is Corollary~\ref{cor:equilateral}: when $q$ is an odd prime power and $3$ is a (non-zero) square in $\Fq$, subsets of $(\mathbb{F}_q^n)^2$ avoiding equilateral triangles are exponentially small.

Our proof of Theorem~\ref{thm:main} hinges on the slice rank polynomial method. This method has been successfully used in many disparate settings to establish upper bound on sizes of sets avoiding certain properties. Croot, Lev and Pach \cite{crootlevpach} used it in a breakthrough result to  show that subsets of $\mathbb{Z}_4^n$ avoiding $3$-term arithmetic progressions are exponentially small. This technique was used by Ellenberg and Gijiswit \cite{capset} to resolve the long-standing Cap Set conjecture, proving that subsets of $\mathbb{F}_3^n$ avoiding $3$-term arithmetic progressions are exponentially small. Observe that both these problems ask for sets avoiding a three-point pattern $\vec{x},\vec{x}+M_1\vec{d},\vec{x}+M_2\vec{d}$ in $X^1$ where $X=\mathbb{Z}_4^n$ or $\mathbb{F}_3^n$ respectively, $M_1=[1]$ and $M_2=[2]$. Ge and Shangguan used slice rank to prove polynomial bounds for subsets of $\Fqn$ avoiding right corners \cite{ge2020maximum}. In addition to the aforementioned work of Kova\u{c} on sets in $(\Fqn)^2$ avoiding right isosceles triangles, the advantage of Theorem~\ref{thm:main} over many of these discoveries is that it unifies several geometric configurations in a universal paradigm. We point to some particular examples, including isosceles triangles with prescribed angles and equilateral triangles, in the discussion after the proof of Theorem~\ref{thm:main}.

We now introduce the slice rank method. Let $X_1,\ldots,X_k$ be sets and $T:X_1 \times \cdots \times X_k \to \mathbb{F}$ for some field $\mathbb{F}$. We say $T$ is a \emph{slice} if it can be written as $T=T' \cdot T''$ where $T':X_i \to \mathbb{F}$ for some $i$ and $T'': X_1 \times \cdots \times X_{i-1} \times X_{i+1} \times \cdots \times X_k \to \mathbb{F}$. The \emph{slice rank} of $T$, denoted $\emph{slice-rank}(T)$ is
\[
\mbox{slice-rank}(T) = \min \left \{r  :  T=\sum_{i=1}^r T^{(i)} \mbox{ where } T^{(i)}:X_1 \times \cdots \times X_k \to \mathbb{F} \mbox{ is a slice for each } i \right \}.
\]
For instance one can show that $T:\mathbb{F}_3^3 \to \mathbb{F}_3$ given by $T(x,y,z)=xy+xz+yz$ is not a slice but since $T(x,y,z)=x(y+z)+y(z)$, it has slice rank $2$. The fundamental benefit of the slice rank construct lies in the following theorem due to Tao:

\begin{theorem}\cite{tao2016notes1}\label{thm:slicerank}
Given a set $A$, let $T:A^k \to \mathbb{F}$ for some field $\mathbb{F}$ and suppose $T(a_1,\ldots,a_k) \neq 0$ if and only if the $a_1=\cdots=a_k$. Then $|A| = \mbox{slice-rank}(T).$
\end{theorem}
If a set of $k$-tuples from a set $A$ avoids a certain property, the slice rank is typically used to bound the size of $A$ in the following way: construct a tensor $T:A^k \to \mathbb{F}$ that is diagonal, and subsequently bound $\mbox{slice-rank}(T)$. In our proof of Theorem~\ref{thm:main} we do this, and show that the natural approach in fact only works for three-point patterns.

\section{Main Result}
We prove our main theorem, Theorem~\ref{thm:main}, and provide examples extending the work of Kova\u{c} \cite{kovac}.
\begin{proof}[Proof of Theorem~\ref{thm:main}]
We proceed with the proof assuming we are given an $r$-point pattern generated by matrices $M_1,M_2,\ldots,M_{r-1}$ with $r \geq 3$, $M_i$ is invertible for each $i$ and $M_i-M_j$ invertible for all $i>j$. We then show that the natural slice rank approach works only when $r=3$. This $r$-point pattern is the set of all distinct $r$-tuples $\vec{x}, \ \vec{x}+M_1 \vec{d}, \ \ldots, \ \vec{x}+M_{r-1} \vec{d}$ with $\vec{x},\vec{d}$ ranging over $(\Fqn)^k$. Throughout the argument we write $\vec{x} \in (\Fqn)^k$ as $\vec{x}=\begin{pmatrix} \vec{x}^{(1)}, \cdots, \vec{x}^{(k)} \end{pmatrix}^T$ where $\vec{x}^{(i)} \in \Fqn$ and has components $\vec{x}^{(i)}(1),\ldots,\vec{x}^{(i)}(n)$. Vectors $\vec{v}_1,\ldots,\vec{v}_r$ belong to the $r$-point pattern generated by $M_1,\ldots,M_{r-1}$ if and only if there exists $\vec{d} \in (\Fqn)^k$ such that $\vec{v}_i-\vec{v}_1 = M_{i-1} \vec{d}$ for each $2 \leq i \leq r$. 
This occurs if and only if $\vec{v}_{i}-\vec{v}_1 = M_{i-1}M_1^{-1}(\vec{v}_2-\vec{v}_1)$ for each $2 \leq i \leq r$. Consider the tensor $T:A^{r} \to \mathbb{F}_q$ given by $\displaystyle T(\vec{v_1},\ldots,\vec{v}_{r}) = \prod_{i=3}^{r} \prod_{j=1}^k \prod_{\ell=1}^n \left(1 -  f_{i,j,\ell}^{q-1}\right)$
where 
\[
f_{i,j,\ell} = \vec{v}_{i}^{(j)}(\ell) - \vec{v}_{1}^{(j)}(\ell) - \left(\sum_{s=1}^k (M_{i-1}M_1^{-1})_{j,s} \cdot \left(\vec{v}_2^{(s)}-\vec{v}_1^{(s)}\right) \right)(\ell) 
\]
Observe that $T$ is a polynomial in the $rkn$-many variables $\{\vec{v}_{i}^{(k)}(\ell)\}_{i,k,\ell}$. Furthermore, $T=1$ if and only if $\vec{v}_{i}-\vec{v}_1 = M_{i-1}M_1^{-1}(\vec{v}_2-\vec{v}_1)$ for each $2 \leq i \leq r$, and $0$ otherwise. Now suppose $T(\vec{v}_1,\ldots,\vec{v}_{r}) \neq 0$. Then $\vec{v}_1,\ldots,\vec{v}_r$ can not be distinct since $A^r$ evades the given $r$-point pattern. So, $\vec{v}_i=\vec{v}_j$ for some $i \neq j$. Letting $M_0$ be the $k \times k$ zero matrix in $\Fq^{k \times k}$, this implies there is a vector $\vec{d}$ for which $\vec{x}+M_{i-1}\vec{d}=\vec{x}+M_{j-1}\vec{d}$. This means $\vec{d} \in \ker(M_{i-1}-M_{j-1})$ but  $M_{i-1}-M_{j-1}$ has trivial kernel so this implies $\vec{d}=\vec{0}$ and hence $\vec{v}_1=\vec{v}_2=\cdots=\vec{v}_{r}$. Altogether then, $T$ is diagonal on $A^{r}$ with nonzero diagonal entries (in fact all diagonal entries are $1$) and therefore $|A|= \mbox{slice-rank}(T)$. To bound $\mbox{slice-rank}(T)$ we start by noticing that $T$ is a degree $kn(r-2)(q-1)$ polynomial in the components of $\vec{v}_1,\ldots,\vec{v}_{r}$. Any monomial in the expansion of $T$ is of the form
\[
C \cdot \prod_{i=1}^r \prod_{j=1}^k \prod_{\ell=1}^n \left(\vec{v}_i^{(j)}(\ell)\right)^{e_{i,j,\ell}}
\]
for some constant $C$ where $1 \leq e_{i,j,\ell} \leq q-1$. We can assume any $e_{i,j,\ell}$ is at most $q-1$ because for all $x \in \Fq$ we have $x^q=x$. Let $e_1,e_2,\ldots,e_r$ be defined by $e_i=\sum_{j=1}^k \sum_{\ell=1}^n e_{i,j,\ell}$. Then there is an $i$ with $1 \leq i \leq r$ so that $e_i \leq kn(r-2)(q-1)/r$ because $\sum_{i=1}^r e_i =kn(r-2)(q-1)$. As a consequence, $\mbox{slice-rank}(T)$ is at most 
\[
r \cdot \left| \left \{ e \in \{0,1,\ldots,q-1\}^{k \times n} \ : \sum_{j=1}^{k} \sum_{\ell=1}^n e_{j,\ell} \leq  \ \dfrac{kn(r-2)(q-1)}{r} \right \} \right|.
\]
Letting $m=\frac{r}{r-2}$, we get that this is at most
\[
r \cdot \min_{0<x<1} \left( \frac{1-x^q}{x^{\frac{q-1}{m}}(1-x)} \right)^{kn}
\]
by standard probabilistic arguments as in Lemma 5 of \cite{naslund2020exponential}.  This quantity is less than $r \cdot c^{kn}$ for some $c$ with $1 \leq c < q$ provided that $m>2$, which occurs only when $r=3$. In this case, the upper bound above is $3 \cdot \min_{0<x<1} \left( \frac{1-x^q}{x^{\frac{q-1}{3}}(1-x)} \right)^{kn}$. Letting $c_q= \min_{0<x<1} \frac{1-x^q}{x^{\frac{q-1}{3}}(1-x)}$, it was proven in Proposition 4.12 of \cite{blasiak2016cap} that $0.8414q \leq c_q \leq 0.9184q$, and the result follows.
\end{proof}
Theorem 3 in \cite{kovac} due to Kova\u{c} follows from our main theorem when $q$ is odd.
\begin{corollary}\label{cor:isos}
For every odd prime power $q$ there is a number $c_q \in (0,q)$ such that the following holds for every positive integer $n$: if a set $A \subseteq (\Fqn)^2$ does not contain a triple of distinct points 
\[
\begin{pmatrix} a \\ b \end{pmatrix}, \ \begin{pmatrix} a+m \\ b+n \end{pmatrix}, \begin{pmatrix} a-n \\ b+m \end{pmatrix}
\]
with $a,b,m,n \in \Fqn$, then its cardinality needs to satisfy the bound $|A| \leq 3 c_q^{2n}$.
\end{corollary}
\begin{proof}
The configuration given is the three-point pattern in $(\Fqn)^2$ generated by the matrices $M_1 = \begin{pmatrix} 1 & 0 \\ 0 & 1 \end{pmatrix}, \ M_2 = \begin{pmatrix} 0 & -1 \\ 1 & 0 \end{pmatrix}$. Since $M_1,M_2$ and $M_1-M_2$ have determinants $1,1,2$ respectively and $q$ is odd, the corresponding matrices are invertible in $\Fq^2$. The result follows by Theorem~\ref{thm:main}.
\end{proof}

We can extend Corollary~\ref{cor:isos} to configurations of triples that have differing geometries by mimicking Euclidean constructions. For instance, suppose $q=7$. Then we can mimic the quantity $\frac{\sqrt{2}}{2}$ by $3 \cdot 2^{-1}=5$ because $3$ is a square root of $2$. Now consider the three-point pattern in $(\Fqn)^2$ generated by $M_1=\begin{pmatrix} 1 & 0 \\ 0 & 1 \end{pmatrix}, \ M_2 = \begin{pmatrix} 5 & -5 \\ 5 & 5 \end{pmatrix}$. The matrix $M_2$ mimics rotation by $45$ degrees. More precisely, if 

\[\begin{pmatrix} x \\ y \end{pmatrix}, \ \begin{pmatrix} x \\ y \end{pmatrix} + M_1 \begin{pmatrix} m \\ n \\ \end{pmatrix}, \begin{pmatrix} x \\ y \end{pmatrix} + M_2 \begin{pmatrix} m \\ n \\ \end{pmatrix}\]
are three points in the three-point pattern then the spread at the first vector is 
\[
1 - \frac{\left(m \cdot (5m-5n)+n \cdot (5m+5n)\right)^2}{\sqrt{m \cdot m + n \cdot n} \sqrt{(5m-5n) \cdot (5m-5n) + (5m+5n) \cdot (5m+5n)}}
\]
which simplifies to $1-5^2$. The spread between two vectors is the finite field analogue of the square of the sine between them (see \cite{lund2018distinct} for instance for details), and so in this way the three-point pattern in $(\mathbb{F}_7^n)^2$ generated by $M_1$ and $M_2$ consists of triples that are  analogues of isosceles triangles where the angle subtended between the sides with equal lengths is $\theta=45$ degrees. By Theorem~\ref{thm:main} we see that subsets avoiding such geometric figures must therefore be exponentially small. We can generalize this beyond $q=7$ to any $q$ for which $2$ is a square. 

We can further extend this geometric analogy to other angles. For example if $q=11$ then $6$ is a square root of $3$, and so rotation by $60$ degrees in $(\mathbb{F}_{11}^n)^2$ can be mimicked by applying the matrix $\begin{pmatrix} 6 & 8 \\ 3 & 6 \end{pmatrix}$. Because of this, subsets of $(\mathbb{F}_{11}^n)^2$ with no equilateral triangles are exponentially small since they must avoid the full-rank three-point pattern generated by $M_1=\begin{pmatrix} 1 & 0 \\ 0 & 1 \end{pmatrix}, \ M_2=\begin{pmatrix} 6 & 8 \\ 3 & 6 \end{pmatrix}$. We prove this for general $q$ when $3$ has a square root in $\Fq$.

\begin{corollary}\label{cor:equilateral}
Let $q$ be an odd prime power and suppose $3$ is a (non-zero) square in $\mathbb{F}_q$. If $A \subseteq (\mathbb{F}_q^n)^2$ contains no equilateral triangle then $|A| \leq 3 c_q^{2n}$ where $0.8414q \leq c_q \leq 0.9184q$.
\end{corollary}
\begin{proof}
 Let $a \in \Fq$ be such that $a^2=3$ and let $b$ be the inverse of $2$. Consider the three-point pattern
\[
\begin{pmatrix} x \\ y \end{pmatrix}, \ \begin{pmatrix} x \\ y \end{pmatrix} +  \begin{pmatrix} 1 & 0 \\ 0 & 1 \end{pmatrix} \begin{pmatrix} m \\ n \end{pmatrix}, \ \begin{pmatrix} x \\ y \end{pmatrix} + \begin{pmatrix} b & -ab \\ ab & b \end{pmatrix} \begin{pmatrix} m \\ n \end{pmatrix}.
\]
First note that this is a full-rank three-point pattern. Indeed the determinant of the latter matrix is $b^2(a^2+1)=1$ and the determinant of the difference of the matrices is $a^2b^2+(b-1)^2=b^2(a^2+1)-2b+1=1$. Now the square of the length between the first and second vector is $m \cdot m + n \cdot n$. The square of the length between the first and third vector is $(bm-abn)\cdot(bm-abn)+(abm+bn)\cdot(abm+bn)=b^2(a^2+1)(m \cdot m+n \cdot n)$ which is $m \cdot m + n \cdot n$ since $b^2(a^2+1)=1$. Finally the square of the length between the second and third vector is $((b-1)m-abn)\cdot ((b-1)m-abn)+(abm+(b-1)n) \cdot (abm+(b-1)n)=(a^2b^2+(b-1)^2)(m \cdot m+n \cdot n)$ which again is $m \cdot m + n \cdot n$ since $a^2b^2+(b-1)^2=1$. Subsequently if $A$ contains no equilateral triangle then it avoids the given full-rank three-point pattern and the result follows by Theorem~\ref{thm:main}.
\end{proof}


\section{Conclusion}

The most natural follow-up question is whether our main theorem extends beyond three-point patterns. In Berger et. al \cite{berger2022popular}, it was shown that popular differences are not guaranteed for four point patterns unless extra conditions are placed on the generating matrices besides being invertible and having their differences invertible. We suspect conditions of a similar kind would be required to assert that sets avoiding $r$-point patterns for $r \geq 4$ are exponentially small. Furthermore, our proof of Theorem~\ref{thm:main} highly suggests that for such patterns, a technique different than the slice-rank method will be needed.

\bibliography{references}{}
\bibliographystyle{plain}

\end{document}